\documentclass[reqno, 12pt, twoside]{amsart}

\usepackage[margin=2.5cm]{geometry}

\usepackage{amsmath,amssymb,amsthm,bm,colonequals,graphicx,mathrsfs,mathtools,microtype,stmaryrd, amsfonts, dsfont,physics, comment, multicol, hyperref, tikz-cd}
\usepackage[shortlabels]{enumitem}
\usepackage[colorinlistoftodos]{todonotes}

\usepackage[style=numeric, sorting=nyt, doi=false,isbn=false,url=false, maxbibnames=99, giveninits=true]{biblatex}
\AtEveryBibitem{\clearfield{number}}
\AtEveryBibitem{\clearfield{eprintclass}}

\renewbibmacro{in:}{%
  \ifentrytype{article}
    {}
    {\bibstring{in}%
     \printunit{\intitlepunct}}}
     \DeclareFieldFormat[thesis]{title}{\textit{#1}}

\DeclareFieldFormat
  [article,inbook,incollection,inproceedings]
  {title}{#1}
  \DeclareFieldFormat
  [misc]
  {title}{#1,\textit{ Preprint}\nopunct}
  \DeclareFieldFormat
  [misc]
  {date}{\mkbibparens{#1},}
  \DeclareFieldFormat
  [article,inbook,incollection,inproceedings,patent,
   unpublished,techreport,misc]
  {volume}{\textbf{#1}}

\numberwithin{equation}{section}

\theoremstyle{plain}

\newtheorem{theorem}{Theorem}[section] 

\newtheorem{lemma}[theorem]{Lemma} 

\newtheorem{proposition}[theorem]{Proposition} 

\newtheorem{proposition-definition}[theorem]{Proposition-Definition} 

\newtheorem{corollary}[theorem]{Corollary}

\theoremstyle{definition}

\newtheorem{definition}[theorem]{Definition}

\newtheorem{example}[theorem]{Example}

\theoremstyle{remark}

\newtheorem{remark}[theorem]{Remark}

\newcommand{\eps}{\varepsilon}

\newcommand{\mcal}{\mathcal}

\newcommand{\NN}{\mathbb{N}}
\newcommand{\ZZ}{\mathbb{Z}}
\newcommand{\QQ}{\mathbb{Q}}
\newcommand{\CC}{\mathbb{C}}

\newcommand{\FF}{\mathbb{F}}
\newcommand{\A}{\mathbb{A}}
\newcommand{\PP}{\mathbb{P}}

\DeclareMathOperator{\cha}{char}

\DeclareMathOperator{\Mor}{Mor}
\DeclareMathOperator{\evalu}{ev}
\def\OK{\mathcal{O}}

\def\RR{\mathbb{R}}

\def\NN{\mathcal{N}}
\def\M{\mathcal{M}}
\DeclareMathOperator{\PGL}{PGL}

\renewcommand{\tilde}{\widetilde}
\renewcommand{\bar}{\overline}

\title{Linear growth and moduli spaces of rational curves}
\author{Jakob Glas}
\date{\today}
\address{Leibniz Universität Hannover\\
Welfengarten 1\\
30167 Hannover\\
Germany}
\email{glas@math.uni-hannover.de}

\begin{document}

\begin{abstract}
    Working in positive characteristic, we show how one can use information about the dimension of moduli spaces of rational curves on a Fano variety $X$ over $\FF_q$ to obtain strong estimates for the number of $\FF_q(t)$-points of bounded height on $X$. Building on work of Beheshti, Lehmann, Riedl and Tanimoto~\cite{BeheshtiLehmannRiedlTanimoto.dP}, we apply our strategy to del Pezzo surfaces of degree at most 5. In addition, we also treat the case of smooth cubic hypersurfaces and smooth intersections of two quadrics of dimension at least 3 by showing that the moduli spaces of rational curves of fixed degree are of the expected dimension. For large but fixed $q$, the  bounds obtained come arbitrarily close to the linear growth predicted by the Batyrev--Manin conjecture. 
\end{abstract}

\maketitle

\thispagestyle{empty}
\setcounter{tocdepth}{1}

\section{Introduction}
For a smooth Fano variety $X$ over a global field $K$ endowed with an anti-canonical height function $H\colon X(K)\to \RR_{\geq 0}$ and a Zariski open subset $U\subset X$, define the counting function
\[
N_U(B)=\#\{x\in U(K)\colon H(x)\leq B\}
\]
for $B>0$. A weak version of the Batyrev--Manin conjecture~\cite{batyrev1990manin} predicts that
\begin{equation}\label{Eq: BatMan}
N_U(B)=O_\eps\left(B^{1+\varepsilon}\right)
\end{equation}
for a suitable open subset $U\subset X$. Suppose that $X\subset \PP^{n-1}$ is defined over a finite field $\FF_q$ and $K=\FF_q(t)$. In this case, a $K$-rational point $x\in X(K)$ corresponds to a morphism $\tilde{x}\colon \PP^1\to X$ defined over $\FF_q$ and we can take $\log_q H(x) = -K_X\cdot \tilde{x}_*\PP^1$. Therefore, we have a natural identification of the set of $x\in X(\FF_q(t))$ such that $H(x)=q^e$ and the set of $\FF_q$-points of the scheme $\Mor(\PP^1,X, e)$, consisting of morphisms $f\colon \PP^1\to X$ such that $-K_X\cdot f_*\PP^1=e$. Via deformation theory~\cite[Theorem 2.6]{DebarreHighdim} one can show that the dimension of every non-empty irreducible component of $\Mor(\PP^1,X,e)$ is at least $e+\dim(X)$ and we refer to this quantity as the \emph{expected dimension}. 

For an open subset $U\subset X$, denote by $\Mor_U(\PP^1,X,e)$ the open subscheme of $\Mor(\PP^1,X,e)$ consisting of morphisms $f\colon \PP^1\to X$ such that $f(\PP^1)\cap U\neq \emptyset$. In this setting (in unpublished notes) Batyrev developed a heuristic that leads to \eqref{Eq: BatMan}, based on the following three assumptions:
\begin{enumerate}
    \item there exists a Zariski open subset $U\subset X$ such that every irreducible component of $\Mor_U(\PP^1,X,e)$ is of the expected dimension,
    \item the number of irreducible components of $\Mor(\PP^1,X,e)$ grows like a polynomial in $e$,
    \item a geometrically irreducible variety $Z$ over $\FF_q$ contains approximately $q^{\dim(Z)}$ points.
\end{enumerate}
In general, (1) may fail if we take $U=X$, which is related to the existence of accumulating subvarieties. By the Lang--Weil estimate, we know that an irreducible variety $Y$ over $\FF_q$ contains $q^{\dim Y}(1+O(q^{-1/2}))$ $\FF_q$-points. Therefore, (3) is true as $q\to\infty$. 

Of course, it is natural to try and see to what extent Batyrev's heuristic can be made rigorous. This would be particularly interesting as it offers a strategy towards the Batyrev--Manin conjecture exploiting the geometric nature of $\FF_q(t)$ that is not available in the integer setting. One can express the number of $\FF_q$-points of $\Mor_U(\PP^1,X,e)$ as the alternating sum of the traces of the action of the Frobenius on the compactly supported cohomology groups of $\Mor_U(\PP^1,X)$ via the Grothendieck--Lefschetz trace formula. This  approach is limited by our knowledge about the compactly supported cohomology groups of $\Mor_U(\PP^1,X,e)$ and understanding them is in general a very interesting, but also difficult problem.

We will not follow this route. Instead, using an elementary approach, we will show that just knowing (1) can already lead to strong \emph{upper bounds} for the number of rational points of bounded height. Building on recent work of Beheshti, Lehmann, Riedl and Tanimoto~\cite{BeheshtiLehmannRiedlTanimoto.dP} on rational curves on del Pezzo surfaces in positive characteristic, we will establish the following result. 
\begin{theorem}\label{Th: MainThmdelPezzo}
    Let $X$ be a smooth del Pezzo surface of degree $d$ over $\FF_q$ with $1\leq d\leq 5$. If we take $U$ to be the complement of the union of all rational curves $C\subset X$ such that $-K_X\cdot C\leq 1$, then 
    \[
    N_U(q^e) = O(C_d^e q^e),
    \]
    where 
    \[
    C_d=\begin{cases}
        1024 &\text{if }d=5,\\
        16 &\text{if }d=4,\\
        27 &\text{if }d=3,\\
        4^4 &\text{if }d=2,\\
        6^6 &\text{if }d=1
    \end{cases}
    \]
    and the implied constant depends on $q$, but not on $e$. 
\end{theorem}
Note that for $d\geq 2$ we are removing all points on lines contained in $X$ and for $d=1$ we additionally remove the singular members of $|-K_X|$.

Currently, we are desperately far from proving \eqref{Eq: BatMan} for del Pezzo surfaces of low degree. In particular, working over $\QQ$, the sharpest estimates available are due to Salberger~\cite{TheSalberger}, showing that $N_U(B)\ll_\eps B^{3/\sqrt{d}+\eps}$ for any del Pezzo surface of degree $d$ over $\QQ$. When $d=3$, in~\cite{SalbergerCubicSurfaces} he is able to improve the exponent $\sqrt{3}$ to $12/7$. For $d=1$, over $\QQ$ we have the better upper bound $N_U(B)\ll B^{2.87}$ that follows from combining work of Bhargava et al ~\cite{bhargavaEll} with work of Helfgott and Venkatesh~\cite{HelfVenkatesh}. Assuming that a cubic surface contains three coplanar lines defined over $\QQ$, we additionally have the estimate $N_U(B)\ll_\eps B^{4/3+\eps}$ due to Heath-Brown~\cite{HeathBrownCubicConicBundle}. If one is willing to assume a deep conjecture relating the rank of an elliptic curve to its conductor, he shows~\cite{heath1998counting} that the same estimate applies to all cubic surfaces over $\QQ$. Hochfilzer and the author~\cite{glas2024rationalpointsdelpezzo} exploited the fact that this conjecture is known over $\FF_q(t)$ thanks to work of Brumer~\cite{Brumerelliptic} and generalising Heath-Brown's approach to arbitrary degree, they showed that $N_U(q^e)=O_\eps (q^{e(1+1/d+\eps)})$ for any smooth del Pezzo surface of degree $d$ over $\FF_q(t)$ with $\cha(\FF_q)>3$. It is clear that if $q$ is sufficiently large with respect to $d$ that the estimate from Theorem~\ref{Th: MainThmdelPezzo} improves upon all these results.

It seems that Theorem~\ref{Th: MainThmdelPezzo} is the first instance confirming that rational points on lines and singular members of $|-K_X|$ dominate the point count when $d=1$. In addition, Theorem~\ref{Th: MainThmdelPezzo} solves an $\FF_q(t)$-version of a problem suggested by Browning~\cite{BrowningdelPezzoSurvey}. In his survey on Manin's conjecture for del Pezzo surfaces, he observes that we do not know of a single example of a smooth cubic surface with a rational point for which $N_U(B)\ll B^{4/3-\theta}$ holds for some $0<\theta < 1/3$ and poses the challenge of providing such a bound. Theorem~\ref{Th: MainThmdelPezzo} produces such a bound for any smooth cubic surface defined over $\FF_q$ as soon as $q>27^{3/(1-3\theta)}$. 

Our proof of Theorem~\ref{Th: MainThmdelPezzo} relies on a simple estimate for the number of $\FF_q$-points on affine varieties in terms of the degree and the dimension, based on B\'ezout's inequality. Combining this with the recent work of Beheshti, Lehmann, Riedl and Tanimoto~\cite{BeheshtiLehmannRiedlTanimoto.dP}, which shows that the only irreducible components of $\Mor(\PP^1,X,e)$ that are not of the expected dimension parameterise multiple covers of lines, yields the desired conclusion. Our strategy is flexible enough to provide upper bounds in other situations. Let $X\subset \PP^{n-1}$ be a smooth Fano variety. For an open subset $U\subset X$, let $\Mor_{e,U}(\PP^1,X)$ be the space of morphisms $\PP^1\to X$ of degree $e$, meaning that $\deg f^*\OK_X(1)=e$, such that $f(\PP^1)\cap U\neq \emptyset$ and define 
\[
N_U(e)=\#\{x\in U(\FF_q(t))\colon H(x)= q^e\},
\]
where $H\colon \PP^{n-1}(\FF_q(t))\to q^\ZZ$ denotes the naive height. If $U=X$, we will simply write $\Mor_e(\PP^1,X)$ instead of $\Mor_{e,U}(\PP^1,X)$. 

\begin{theorem}\label{Thm: RatPointsCompInt}
    Let $X\subset \PP^{n-1}$ be a smooth Fano variety defined by forms of degree $d_1,\dots,  d_R$ over $\FF_q$.  Then,
    \[
    N_U(e)=O(C^e q^{\dim \Mor_{e,U}(\PP^1,X)}),
    \]
    where 
    \begin{equation}\label{Defi: C.Constant}
    C=\prod_{i=1}^R d_i^{d_i}.
    \end{equation}
\end{theorem}
In particular, if $\Mor_{e,U}(\PP^1,X)$ is of the expected dimension for $e\gg1$, then by taking $q$ to be large, we again come arbitrarily close to the growth predicted by \eqref{Eq: BatMan}.

Unfortunately, our understanding of the moduli spaces of rational curves in positive characteristic is not as far developed as in characteristic 0, thus limiting the applications of Theorem~\ref{Thm: RatPointsCompInt}. To provide more examples, we shall prove the following result, which extends work of Coskun and Starr~\cite{CoskunStarr} on cubic hypersurfaces and of Okamura~\cite{Okamura.dpManifolds} on intersections of two quadrics to positive characteristic.
\begin{theorem}\label{Th: ModuliSpacescubicQuadric}
    Let $X\subset \PP^{n-1}$ be a smooth cubic hypersurface or a smooth complete intersection of two quadrics over a field $K$. Assume that $\cha(K)>3$ if $X$ is a cubic hypersurface. If $\dim(X)\geq 3$,
    then every irreducible component of $\Mor_e(\PP^1,X)$ is of the expected dimension 
    \[
\begin{cases}
    e(n-4)+n-3 &\text{if }X\text{ is an intersection of two quadrics},\\
    e(n-3)+n-2 &\text{if }X\text{ is a cubic hypersurface},
\end{cases} 
    \]
    for every $e\geq 1$.
\end{theorem}
In view of Theorem~\ref{Thm: RatPointsCompInt}, we immediately deduce the following estimate.
\begin{theorem}
    Under the assumptions of Theorem~\ref{Th: ModuliSpacescubicQuadric}, we have 
    \[
    N_X(e)=O(16^eq^{e(n-4)})
    \]
    for any smooth complete intersection $X\subset \PP^{n-1}$ of two quadrics over $\FF_q$ and 
    \[
    N_X(e)=O(27^eq^{e(n-3)})
    \]
    for any smooth cubic hypersurface $X\subset \PP^{n-1}$ over $\FF_q$.
\end{theorem}
When $X=V(F_1,\dots, F_R)\subset \PP^{n-1}$ is a smooth complete intersection of dimension $n-1-R$, the adjunction formula implies that the divisor $-K_X$ corresponds to the line bundle $\OK_X(n-d_1-\cdots -d_R)$. In particular, in agreement with our result, \eqref{Eq: BatMan} predicts that $N_X(e)=O_\varepsilon(q^{e(n-4+\varepsilon)})$ when $X\subset \PP^{n-1}$ is a smooth complete intersection of two quadrics and $N_X(e)=O_\varepsilon(q^{e(n-3+\varepsilon)})$ when $X\subset \PP^{n-1}$ is a smooth cubic hypersurface.

For comparison, we know an asymptotic formula for a weighted version of $N_X(q^e)$ for any smooth hypersurface of dimension $6$ over $\FF_q(t)$ thanks to work of Browning and Vishe~\cite{CubicHypersurfacesBV}. For $n=7$, the same is known if we assume the defining equation $F$ to be diagonal, by work of Hochfilzer and the author~\cite{glas2022question}. Again assuming $F$ to be diagonal, in they same work they establish an upper bound of the form $N_F(q^e)\ll_\eps q^{e(3+\eps)}$ for $n=6$. 
All these works use the circle method and it seems that without the injection of radical new ideas, we cannot do better than $N_X(q^e)\ll q^{9e/4}$ for $n=5$ using the circle method alone. Even the  implementation of the ratios conjecture into the circle method by Wang~\cite{wang2021approaching}, which was carried over to function fields by Browning, Wang and the author~\cite{browningGlasWang2024sums}, might only be able to produce the conditional estimate $q^{e(9/4-\delta)}$ for some tiny $\delta>0$. 

If $X$ is a Fano variety defined over $\FF_q(t)$, then one can still parameterise the set of $\FF_q(t)$-rational points of bounded height by a suitable space of sections of fixed degree. It would be interesting to see whether one can also obtain dimension bounds for these spaces, as then our strategy would imply estimates for varieties that are not necessarily defined over the constant field.

\subsection*{Acknowledgements}
The author would like to thank Tim Browning, Raymond Cheng, Matthew Hase-Liu, Eric Riedl and Victor Wang for helpful comments and discussions. 
\section{Uniformly counting rational points over finite fields}
By convention, a variety is a closed subset of affine or projective space and irreducible will always mean geometrically irreducible.

Let $N\geq 1$ and $Y\subset \A^N$ be a non-empty variety defined over $\bar{\FF}_q$. 
We define the degree $\deg(Y)$ of $Y$ to be the sum of the degrees of its irreducible components. We will write $\dim(Y)$ for the maximal dimension of the irreducible components of $Y$. In particular, our definition of degree takes into account irreducible components of $Y$ of dimension smaller than $\dim(Y)$. By convention, we set $\deg(\emptyset)=0$. We extend both the degree and the dimension to locally closed subsets of affine space by defining them to be the degree and dimension of their closures respectively.


Our main technical input for Theorems~\ref{Th: MainThmdelPezzo} and \ref{Thm: RatPointsCompInt} is the following simple estimate.
\begin{proposition}\label{Prop: UniformPointCount}
    Let $Y\subset \A^N$ be a locally closed subset defined over $\bar{\FF}_q$. Then 
    \[
    \#Y(\FF_q)\leq \deg(Y)q^{\dim(Y)}.
    \]
\end{proposition}
We do not claim that Proposition~\ref{Prop: UniformPointCount} is new. In fact, proofs have appeared in~\cite{FiniteFieldCount}, \cite{EllenbergOberlinTaoKakeya}, \cite{DvirKollarVarietyEvasisve} and \cite{browningShuntaro2024rational}. To be self-contained, we will present a full proof. 

Our proof of Proposition~\ref{Prop: UniformPointCount} uses an induction procedure, based on the following version of B\'ezout's inequality. A proof can for example be found in~\cite[Theorem 1]{HeintzDefinability}.
\begin{lemma}\label{Le: BezoutHyper}
Let $Y\subset \A^N$ be a hypersurface and $Z\subset \A^N$ a variety. Then $\deg(Y\cap Z)\leq \deg(Y)\deg(Z)$. 
\end{lemma}
The following lemma will also be useful at several points.
\begin{lemma}\label{Le: DeltaonOpen}
    Let $Y\subset \A^N$ be locally closed and $U\subset Y$ be open. Then $\deg(U)\leq \deg(Y)$.
\end{lemma}
\begin{proof}
    This follows from the observation that every irreducible component of the closure of $U$ is also an irreducible component of the closure of $Y$ and that the degree of a locally closed subset is the sum of the degrees of the irreducible components of its closure.
\end{proof}
\begin{proof}[Proof of Proposition~\ref{Prop: UniformPointCount}]
    As both the dimension and the degree of $Y$ and its closure agree and the number of $\FF_q$-points of $Y$ is certainly at most the number of $\FF_q$-points of its closure, we may without loss of generality assume that $Y$ is closed.
    
    We proceed via an induction on $\dim(Y)$. If $\dim(Y)=0$, then $Y(\FF_q)$ is a union of at most $\deg(Y)=\deg(Y)$ points, so that the statement holds trivially.

    Let us suppose now that $\dim(Y)>0$ and write $Y=Z_1 \cup Z_2$, where $Z_1$ is the union of the irreducible components of $Y$ of dimension strictly smaller than $\dim(Y)$ and $Z_2$ is the union of irreducible components of $Y$ of dimension $\dim(Y)$. Note that in particular $\deg(Y)=\deg(Z_1)+\deg(Z_2)$, which follows directly from the definition of the degree. Then by the induction hypothesis we have $\#Z_1(\FF_q)\leq \deg(Z_1)q^{\dim(Z_1)}\leq \deg(Z_1)q^{\dim(Y)}$. 

    As we have $\dim(Z_2)>0$, we can find an index $1\leq i_0\leq N$ such that the hyperplane $H_a\subset \A^N$ defined by $x_{i_0}=a$ intersects $Z_2$ properly for any $a\in\bar{\FF}_q$. Thus 
    \begin{align*}
        \#Z_2(\FF_q) &\leq\sum_{a\in \FF_q}\#(Z_2\cap H_a)(\FF_q)\\
        &\leq \sum_{a\in \FF_q} \deg(Z_2\cap H_a)q^{\dim(Z_2\cap H_a)}\\
        &\leq \deg(Z_2)\sum_{a\in \FF_q}\deg(H_a)q^{\dim(Z_2)-1}\\
        &\leq \deg(Z_2)q^{\dim(Y)},
    \end{align*}
    where second inequality follows from the induction hypothesis; the third from Lemma~\ref{Le: BezoutHyper} and $\dim(Z_2\cap H_a)<\dim(Z_2)$; and the fourth from the fact that $\deg(H_a)=1$ for any $a\in \FF_q$ and $\dim(Z_2)=\dim(Y)$. Therefore, we have 
    \[
    \#Y(\FF_q)\leq \#Z_1(\FF_q)+\#Z_2(\FF_q)\leq (\deg(Z_1)+\deg(Z_2))q^{\dim(Y)}=\deg(Y)q^{\dim(Y)}.
    \]
\end{proof}
\section{Counting $\FF_q$-points on moduli spaces of rational curves}
Our goal is now to apply Proposition~\ref{Prop: UniformPointCount} to prove Theorems~\ref{Th: MainThmdelPezzo} and \ref{Thm: RatPointsCompInt}. We will begin with the latter. 

Suppose that $X\subset \PP^{n-1}$ is defined as the common vanishing locus of forms $F_1,\dots, F_R\in \FF_q[x_1,\dots, x_n]$ of degrees $d_1,\dots, d_R$. Let $\bm{x}(t)\in\FF_q[t]^n$ be such that $\max_{i=1,\dots, n}\deg(x_i)=e$. Then we can write 
\[
F_i(\bm{x}(t))=\sum_{j=0}^{d_ie}f_{ij}(\bm{x})t^j
\]
for $i=1,\dots, R$, where each $f_{ij}$ is a form over $\FF_q$ of degree $e$ in the coefficients of the coordinates of $\bm{x}$. Identifying $\bm{x}$ with its coefficient vector, the equations $f_{ij}=0$ for $1\leq i \leq R$ and $0\leq j\leq d_ie$ define a subvariety $Y_e$ inside $\PP^{n(e+1)-1}$. The space of morphisms $\Mor_e(\PP^1,X)$ is then the intersection of an open subset $U_e\subset \PP^{n(e+1)-1}$ with $Y_e$. The open set $U_e$ captures that we require $\max_{i=1,\dots, n}\deg(x_i)=e$ and that $\gcd(x_1(t),\dots, x_n(t))=1$ in $\FF_q[t]$, which can be seen to be a Zariski open condition using resultants. 

\begin{proposition}\label{Prop: GeneralCount}
    In this setting, suppose that $Z\subset \Mor_e(\PP^1,X)$ is a union of irreducible components of $\Mor_e(\PP^1,X)$. Then for 
    \[
    C=\prod_{i=1}^R d_i^{d_i},
    \]
    we have
    \[
    \#Z(\FF_q) \leq 2d_1\cdots d_R C^e q^{\dim(Z)}.
    \]
\end{proposition}
\begin{proof}
    For a projective variety $Y$, let $\tilde{Y}$ be its affine cone. Let $X_e=\Mor_e(\PP^1,X)$. As every irreducible component of $\tilde{Z}$ is an irreducible component of $\tilde{X}_e$, it is clear from the definition of the degree that $\deg(\tilde{Z})\leq \deg(\tilde{X}_e)$. As $X_e$ is an open subset of $Y_e$, the same applies to their affine cones. Therefore, from Lemmas~\ref{Le: DeltaonOpen} and ~\ref{Le: BezoutHyper} we obtain 
    \[
    \deg(\tilde{Z})\leq \deg(\tilde{X}_e)\leq \deg(\tilde{Y}_e)\leq \prod_{i=1}^R\prod_{j=0}^{d_ie}d_i=\prod_{i=1}^Rd_i^{d_ie+1}.
    \]
    As the number of $\FF_q$-points on $Z$ is certainly at most $1/(q-1)\leq 2/q$ times the number of $\FF_q$-points on its affine cone and $\dim(\tilde{Z})=\dim(Z)+1$, we thus conclude from Proposition~\ref{Prop: UniformPointCount} that 
    \begin{align*}
    \#Z(\FF_q)\leq 2\#\tilde{Z}(\FF_q)/q 
    \leq 2\deg(\tilde{Z})q^{\dim(Z)}
    \leq 2\left(\prod_{i=1}^Rd_i^{d_ie+1}\right)q^{\dim(Z)}.        
    \end{align*}
\end{proof}
\begin{proof}[Proof of Theorem~\ref{Thm: RatPointsCompInt}]
Let $Z$ be the closure of $\Mor_{e,U}(\PP^1,X)$ inside $\Mor_e(\PP^1,X)$. Then Proposition~\ref{Prop: GeneralCount} gives
    \begin{align*}
        N_U(e)& = \#\Mor_{e,U}(\PP^1, X)(\FF_q)\\
        &\ll C^eq^{\dim\Mor_{e,U}(\PP^1,X)}
    \end{align*}
    where the implied constant only depends on $d_1,\dots, d_R$ and $q$. This completes the proof of Theorem~\ref{Thm: RatPointsCompInt}.
\end{proof}

Next we turn to del Pezzo surfaces. Let $X$ be a del Pezzo surface. Our main ingredient in is the following result of Beheshti, Lehmann, Riedl and Tanimoto~\cite[Theorem 1.1]{BeheshtiLehmannRiedlTanimoto.dP}. 
\begin{theorem}\label{Th: Sho}
    Let $X$ be a smooth del Pezzo over $\FF_q$. Then for any $e\geq 1$, the only irreducible components of $\Mor(\PP^1,X,e)$ that are not of the expected dimension are those parameterising multiple covers of rational curves $C\subset X$ such that $-K_X\cdot C \leq 1$.
\end{theorem}
Recall that here the $e$ in $\Mor(\PP^1,X,e)$ now refers to the anti-canonical degree as opposed to our more general setting above, where in $\Mor_e(\PP^1,X)$ it referred to the degree under the given embedding into $\PP^n$.  
\begin{remark}
In fact, Theorem 1.1 of~\cite{BeheshtiLehmannRiedlTanimoto.dP} is stated more generally for weak del Pezzo surfaces, which additionally requires the assumption that a general member of $|-K_X|$ is smooth. This is always satisfied for smooth del Pezzo surfaces, as is explained in~\cite[Section 3]{BeheshtiLehmannRiedlTanimoto.dP}. In addition, they do not work with $\Mor(\PP^1,X,e)$ but instead with the space of stable maps generically parameterising morphisms with irreducible domain inside the coarse moduli space of the Kontsevich space, which immediately implies the same result for the morphism space.    
\end{remark}
Let $U\subset X$ be the Zariski open subset defined as the complement of all rational curves $C\subset X$ satisfying $-K_X\cdot C \leq 1$, so that rational points $x\in U(\FF_q(t))$ satisfying $H(x)=q^e$ are in bijection with the $\FF_q$-points of $\Mor_U(\PP^1,X,e)$. 

First, let us assume that $3\leq d\leq 5$. In this case the anti-canonical divisor induces an embedding $\varphi\colon X\to \PP^d$. If $d=5$, then $X$ can be realised as an intersection of 5 quadrics; if $d=4$ the image is a complete intersection of 2 quadrics and if $d=3$ it is a cubic surface inside $\PP^3$.

Let $F_1,\dots, F_{R_d}\in \FF_q[x_1,\dots,x_{d+1}]$ be the defining equation for $X$ in $\PP^d$. As $\varphi^*(\OK_{\PP^d}(1))=\omega_X^{-1}$, in the notation of the introduction we thus have $\Mor_e(\PP^1,X)=\Mor(\PP^1,X,e)$. Let $Z$ be the closure of $\Mor_U(\PP^1,X,e)$ inside $\Mor_e(\PP^1,X)$. Then Theorem~\ref{Th: Sho} guarantees that every irreducible component of $Z$ has dimension $e+2$. In particular, it follows from Proposition~\ref{Prop: GeneralCount} that 
\[
\# Z(\FF_q) \ll C_d^e q^e,
\]
with 
\[
C_d=\begin{cases}
    1024=4^5&\text{if }d=5,\\
    16=4^2 &\text{if }d=4,\\
    27=3^3 &\text{if }d=3.
\end{cases}
\]
Therefore, we have 
\[
N_U(q^e)\leq \sum_{k=0}^e \#\Mor_U(\PP^1,X,k) \ll \sum_{k=0}^e (C_d q)^k \ll (C_d q)^e,
\]
which completes the proof of Theorem~\ref{Th: MainThmdelPezzo} for $3\leq d \leq 5$. 

It will be convenient in what follows to write $\PP(e_1^{f_1},\dots, e_r^{f_r})$ to denote the weighted projective space $\PP(e_1,\dots, e_1, \dots, e_r,\dots, e_r)$, where each $e_i$ appears precisely $f_i$-times. In addition, if $Y\subset \PP(e_1^{f_1},\dots, e_r^{f_r})$ is locally closed, then we we denote by $\tilde{Y}\subset \A^{f_1+\cdots +f_r}$ its affine cone. For example, if $Y$ is defined by some equations $G_1=\cdots = G_s=0$, then $\tilde{Y}$ is the variety inside $\A^{e_1+\cdots +e_r}$ defined by the same equations. 

Next, let us assume that $d=2$. In this case, we can realise $X$ as quartic hypersurface inside $\PP(2,1,1,1)$ that is defined by an equation of the shape $F(y,u,v,w)=0$, where $F(y,u,v,w)=y^2-f(u,v,w)$ and $f\in \FF_q[u,v,w]$ is a binary quartic form.

Let $(y,u,v,w)\in \FF_q[t]^4$, where $\max\{\deg y/2, \deg u, \deg v, \deg w\}=e$ and write 
\[
F(y,u,v,w)=\sum_{i=0}^{4e}f_i(y,u,v,w)t^i,
\]
where each $f_i$ is a polynomial in the coefficients of $y,u,v,w$ with coefficients in $\FF_q$ that is homogeneous of degree 2 in the coefficient vector of $y$ and homogeneous of degree 4 in the coefficient vector of $(u,v,w)$. With this description, we can realise $X_e=\Mor(\PP^1,X,e)$ again as an open subscheme of $Y_e\subset \PP(2^{e+1}, 1^{e+1}, 1^{e+1}, 1^{e+1})$, which is defined as the common zero locus of $f_0=\cdots = f_{4e}=0$.  Let $Z$ be the closure of $\Mor_U(\PP^1,X,e)$ inside $X_e$. Then as every irreducible component of $\tilde{Z}$ is an irreducible component of $\tilde{X}_e$, we have $\deg(\tilde{Z})\leq \deg(\tilde{X}_e)$. In addition, $\tilde{X}_e$ is an open subset of $\tilde{Y}_e$, so that by Lemmas~\ref{Le: DeltaonOpen} and \ref{Le: BezoutHyper} we have 
\[
\deg(\tilde{Z})\leq \deg(\tilde{X}_e)\leq \deg(\tilde{Y}_e)\leq \prod_{i=0}^{4e}\deg(f_i)=4^{4e+1}.
\]

Theorem~\ref{Th: Sho} guarantees that $\dim(\tilde{Z})=\dim(Z)+1=e+3$. As $\#Z(\FF_q)\leq \# \tilde{Z}(\FF_q)$, Proposition~\ref{Prop: UniformPointCount} hands us the estimate 
\[
\#\Mor_U(\PP^1,X, e) \leq \#\tilde{Z}(\FF_q) \leq \deg(\tilde{Z})q^{\dim(\tilde{Z})}\leq 4^{4e+1}q^{e+3}.
\]
Hence we obtain 
\begin{align*}
    N_U(q^e)\leq \sum_{k=0}^e \#\Mor_U(\PP^1,X,k) \ll \sum_{k=0}^e (C_2q)^k\ll (C_2q)^e,
\end{align*}
where $C_2=4^4$. This proves Theorem~\ref{Th: MainThmdelPezzo} when $d=2$. 

Finally, we deal with the case $d=1$. In this case $X$ can be realised as a sextic hypersurface inside $\PP(3,2,1,1)$. A similar argument as for $d=2$ shows that we can realise $X_e=\Mor (\PP^1,X,e)$ as an open subscheme of $\PP(3^{3e+1}, 2^{2e+1}, 1^{e+1}, 1^{e+1})$. Moreover, it is contained as an open subscheme of a variety $Y_e$ defined as the common vanishing locus of polynomials $f_0,\dots, f_{6e}$ in $7e+4$ variables that are homogeneous of degree $2$ in the first $3e+1$ variables; homogeneous of degree $3$ in the next $2e+1$ variables and homogeneous of degree 6 in the remaining variables. Let $Z\subset \PP(3^{3e+1}, 2^{2e+1}, 1^{e+1}, 1^{e+1})$ be the closure of $\Mor_U(\PP^1,X,e)$ inside $X_e$. Then as every irreducible component of $\tilde{Z}$ is an irreducible component of $\tilde{X}_e$, we have $\deg(\tilde{Z})\leq \deg(\tilde{X}_e)$. In addition, $\tilde{X}_e$ is open in $\tilde{Y}_e$, so that Lemma~\ref{Le: DeltaonOpen} yields
\[
\deg(\tilde{Z})\leq \deg(\tilde{X}_e) \leq \deg(\tilde{Y}_e)\leq \prod_{i=0}^{6e}\deg(f_i)= 6^{6e+1},
\]
where the penultimate inequality is a consequence of Lemma~\ref{Le: BezoutHyper}. By Theorem~\ref{Th: Sho}, we have $\dim(\tilde{Z})=\dim(Z)+1=e+3$. Therefore, Proposition~\ref{Prop: UniformPointCount} implies
\[
\#\Mor_U(\PP^1,X,e)(\FF_q)\leq \#\tilde{Z}(\FF_q) \leq \deg(\tilde{Z})q^{\dim(\tilde{Z})}\leq 6^{6e+1}q^{e+3}.
\]
Thus we obtain 
\begin{align*}
N_U(q^e)\leq \sum_{k=0}^e\#\Mor_U(\PP^1,X,k)(\FF_q) \ll \sum_{k=0}^e (C_1q)^k\ll (C_1q)^e,
\end{align*}
where $C_1=6^6$. This finishes the case $d=1$ of Theorem~\ref{Th: MainThmdelPezzo} and thus completes its proof.
\begin{remark}
    Alternatively, we could have avoided working with weighted projective space by using a very ample multiple of $-K_X$ to obtain an embedding into projective space.
\end{remark}
\section{Fano scheme of lines}\label{Se: Linescubic}
Throughout this section $K$ denotes an algebraically closed field of any characteristic, unless specified otherwise. Given a variety $X\subset \PP^n$ over $K$, we let $F_1(X)=\{L\in \mathbb{G}(1,n)\colon L\subset X\}$ be the \emph{Fano scheme of lines} of $X$ and for $x\in X$ we denote by $F_1(X,x)\subset F_1(X)$ the subscheme parameterising lines $L\subset X$ containing $x$. When $X\subset \PP^n$ is a cubic hypersurface with $n\geq 3$, we know from work of Altman and Kleiman~\cite{AltKleinFano} that $F_1(X)$ is smooth of dimension $2n-4$.

We begin with the following result about lines on smooth intersections of two quadrics.    
\begin{proposition}\label{Prop: LinesIntersectionQuadrics}
    Let $X=V(Q_1,Q_2)\subset \PP^{n}$ be a smooth intersection of two quadrics. If $n\geq 5$, then $\dim F_1(X,x)= n-5$ for every $x\in X$. 
\end{proposition}
\begin{proof}
    After a suitable change of variables we may assume that $x=(1\colon 0 \colon \cdots \colon 0)$ and that the embedded tangent space $T_xX\subset \PP^n$ is given by $x_2=x_3=0$. This implies that we can take $Q_1$ and $Q_2$ to be of the shape 
    \[
    Q_1(\bm{x})= x_0x_1+q_1(x_1,\dots, x_n) \quad\text{and}\quad Q_2(\bm{x})=x_0x_2+q_2(x_1,\dots, x_n)
    \]
    for some quadratic forms $q_1,q_2\in K[x_1,\dots, x_n]$. A line through $x$ is then of the form $\ell = \{(s\colon a_1t\colon \cdots \colon a_n)\in \PP^n\colon (s\colon t)\in \PP^1\}$, for $(a_1\colon \cdots \colon a_n)\in \PP^{n-1}$. The Fano scheme of lines is now given by $F_1(X,x)= V(a_1,a_2, q_1(a_1,\dots, a_n),q_2(a_1,\dots, a_n))$. In particular, it will be of dimension $n-5$ unless $q_1(0,0,a_3,\dots, a_n)$ and $q_2(0,0,a_3,\dots, a_n)$ share a common factor. If they share a linear factor, say $L$, then $X$ contains the linear space defined by $x_1=x_2=L=0$, which has dimension at least $n-3$. However, the Lefschetz hyperplane theorem (which is also valid in positive characteristic when using \'etale cohomology) implies that the degree of any irreducible subvariety $Z\subset X$ with $\dim(Z)>\dim(X)/2$ is divisible by 4. We have $n-3>\dim(X)/2=(n-2)/2$ for $n\geq 5$ and so this case is impossible. Similarly, if $q_1(0,0,a_3,\dots, a_n)$ is a constant multiple of $q_2(0,0,a_3,\dots, a_n)$, then $X$ contains the $(n-3)$-dimensional quadric $V(x_1,x_2, q_1)$, which is again impossible by the Lefschetz hyperplane theorem when $n\geq 5$.   
\end{proof}
For the remainder of this section $X\subset \PP^n$ denotes a smooth cubic hypersurface over $K$ defined by a form $F\in K[x_0,\dots, x_n]$. The following result is analogous to \cite[Lemma 2.1]{CoskunStarr}. 
\begin{proposition}\label{Le: DimF1(X,x)Cubic}
    Let $X\subset \PP^n$ be a smooth cubic hypersurface and $n\geq 4$. Then there exists an open subset $U\subset X$ such that $\dim F_1(X,x)=n-4$ for any $x\in U$. In addition, if $x\not\in U$, then every irreducible component of $F_1(X,x)$ has dimension $n-3$.
\end{proposition}
\begin{proof}
    After a suitable change of variables we may assume that $x=(1\colon 0 \colon \cdots \colon 0)$ and that the embedded tangent space of $X$ at $x$ is given by $x_2=0$. This implies that $F$ is of the form 
    \[
    F(\bm{x})= x_0^2x_1+ x_0q(x_1,\dots, x_n)+c(x_1,\dots, x_n),
    \]
    where $q,c\in K[x_1,\dots, x_n]$ are a quadratic and a cubic form respectively. As in the proof of Proposition~\ref{Prop: LinesIntersectionQuadrics}, this implies that $F_1(X,x)$ is given by $V(x_1, q,c)\subset \PP^{n-1}$. In particular, $F_1(X,x)$ will be of dimension $n-4$ as long as $q(0,x_2,\dots, x_n)$ and $c(0,x_2,\dots, x_n)$ do not share a common factor. First, suppose $c(0,x_2,\dots, x_n)=L(x_2,\dots, x_n)q(0,x_2,\dots, x_n)$ for some linear form $L\in K[x_2,\dots, x_n]$. Then $X$ contains the $(n-2)$-dimensional linear space $V(x_0, L)$. However, since $X$ is a smooth hypersurface, by \cite[Appendix]{StarrAppendix} the maximal dimension of a linear subspace contained in it is $\lfloor (n-1)/2\rfloor$, which is strictly smaller than $n-2$ for $n\geq 4$. If $c(0,x_2,\dots, x_n)$ vanishes identically, then $X$ would contain the linear subspace $V(x_0,x_1)$ which is again a contradiction. In particular, the only remaining possibility for $V(x_1,q,c)$ to fail to have dimension $n-4$ is when $q(0,x_2,\dots, x_n)$ vanishes identically. Since we have already seen that $c(0,x_2,\dots, x_n)$ cannot vanish identically, we have dim $F_1(X,x)=\dim V(x_1,C)=n-3$ in this case as claimed. Note that in this case $X\cap T_xX$ is the conical cubic hypersurface $V(x_1,c)\subset \PP^{n}$ with vertex $x$. In addition, it is clear that the collection of $x\in X$ for which $X\cap T_xX$ is a cone with vertex $x$ is Zariski closed. 
\end{proof}
\begin{definition}
    Suppose that $n\geq 3$. We say that $x\in X$ is an \emph{Eckardt point} if $X\cap T_xX$ is a cone with vertex $x$.
\end{definition}
Note that when $n\geq 4$, then $x\in X$ is an Eckardt point if and only if $\dim F_1(X,x)=n-3$, while when $n=3$ it is equivalent to $x$ lying on the intersection of three coplanar lines.

It follows from Lemma~\ref{Le: DimF1(X,x)Cubic} that the collection of Eckardt points is a Zariski closed subset of $X$. For our application to rational curves it is crucial to know that there are only finitely many Eckardt points on a smooth cubic hypersurface. Working in characteristic $0$, Coskun and Starr show this with the following argument: if $x$ is a Eckardt point, then the differential of the Gauss map into the dual variety of $X$ vanishes at that point. In particular, if a positive dimensional family of Eckardt points would exist, then the Gauss map would be constant along this family. This, however, would contradict the finiteness of the Gauss map. While the finiteness of the Gauss map still holds in positive characteristic, the vanishing of the differential is not enough to conclude that the Gauss map is constant and so we have to use a different approach. 

Instead, we will relate Eckardt points to special lines on cubic hypersurfaces. Given a smooth variety $Y$ and a smooth subvariety $Z\subset Y$, we let $\NN_{Z/Y}$ be the normal bundle of $Z$ in $Y$. 
\begin{definition}
    Let $L\subset X$ be a line and $n\geq 4$. We say that $L$ is of \begin{enumerate}[(i)]
        \item \emph{type I} if $\NN_{L/X}\simeq \OK_L^2\oplus \OK_L(1)^{n-4}$ and of 
        \item \emph{type II} if $\NN_{L/X}\simeq \OK_L(-1)\oplus \OK_L(1)^{n-3}$.
    \end{enumerate}
\end{definition}
It is known \cite[Lemma 2.7]{FukasawaGaussRank0} that every line on a cubic hypersurface is either of type I or II. 
Let $(\PP^n)^*$ be the dual projective space and $X^*\subset (\PP^n)^*$ the dual variety of $X$. Let $\gamma \colon X \to (\PP^n)^*$ be the Gauss map defined by the assignment $x\mapsto T_xX$, where $T_xX$ denotes the embedded tangent space of $X$ at $x$. By definition, $X^*$ is the image of $\gamma(X)$ in $(\PP^n)^*$ and a hyperplane $H\in (\PP^n)^*$ intersects $X$ tangentially if and only if $H\in X^*$.
\begin{lemma}\label{Le: Existence.SmoothHyperplane}
    Let $L\subset X$ be a line and suppose that $n\geq 4$. Then there exists a hyperplane $H\in (\PP^n)^*$ such that $L\subset H$ and $X\cap H$ is smooth.
\end{lemma}
\begin{proof}
    Let $W\subset (\PP^n)^*$ be the linear subspace of dimension $n-2$ parameterising hyperplanes $H\subset \PP^n$ containing $L$. To prove the lemma, it suffices to show that $W\not\subset X^*$. After a suitable change of variables we may assume that $W=V(y_0,y_1)\subset (\PP^n)^*$, so that $\gamma^{-1}(W)=V(\frac{\partial F}{\partial y_0},\frac{\partial F}{\partial x_1})$. Now if $W\subset X^*$ would hold, then $X=\gamma^{-1}(X^*)$ implies that $\gamma^{-1}(W)\subset X$. In particular, we can write $F=L_0\frac{\partial F}{\partial x_0 }+L_1\frac{\partial F}{\partial x_1 }$ for some linear forms $L_0,L_1\in K[x_0,\dots, x_n]$ and thus $X$ contains the linear subspace $V(L_0,L_1)$, which has dimension at least $n-2$. However, by \cite[Appendix]{StarrAppendix} a smooth hypersurface in $\PP^n$ can only contain linear subspaces of dimension at most $\lfloor (n-1)/2\rfloor$. Hence $n-2\leq (n-1)/2$ must hold, which implies $n\leq 3$. 
\end{proof}
\begin{lemma}\label{Le: EckardtisEckardtonHyperplane}
    Suppose that $n\geq 4$ and $x\in X$ is an Eckardt point. If $H\in (\PP^n)^*$ is a hyperplane containing $x$ such that $H\cap X$ is smooth, then $x$ is also an Eckardt point of $X\cap H$. 
\end{lemma}
\begin{proof}
    Note that since $H\cap X$ is smooth, we have $T_x(X\cap H)=T_xX\cap H$. Let $y\in (X\cap H)\cap (T_x(X\cap H))$ be a point distinct from $x$ and $L$ be the line spanned by $x$ and $y$. We must then have $y\in X\cap T_xX$ and because $x$ is a vertex of the cone $X\cap T_xX$, the line $L$ is contained in $X\cap T_xX$. On the other hand, since both $x$ and $y$ are contained in $H$, $L$ must also be contained in $H$. In particular, $L$ is contained in $X\cap T_xX\cap H = (X\cap H)\cap T_x(X\cap H)$. This shows that the line spanned by $x$ and an arbitrary point $y\in (X\cap H)\cap T_x(X\cap H)$ distinct from $x$ is contained in $(X\cap H)\cap T_x(X\cap H)$ and hence $(X\cap H)\cap T_x(X\cap H)$ is a cone with vertex $x$. This means that $x$ is an Eckardt point of $X\cap H$. 


\end{proof}
\begin{lemma}\label{Le: Line.only2Eckardt}
    Let $n\geq 3$. If $\cha(K)\neq 2$, then any line $L\subset X$ contains at most two Eckardt points. If $\cha(K)=2$, then any line $L\subset X$ contains at most $5$ Eckardt points. 
\end{lemma}
\begin{proof}
    We proceed by induction on $n$. The case $n=3$ is the content of Lemma~9.4 in \cite{DolgachevAutoCubicPositiveCha}. So suppose that $n>3$. By Lemma~\ref{Le: Existence.SmoothHyperplane} there exists a hyperplane $H\subset \PP^n$ containing $L$ such that $H\cap X$ is a smooth cubic hypersurface. Lemma~\ref{Le: EckardtisEckardtonHyperplane} tells us that $x$ is also an Eckardt point of $X\cap H$. The claim thus follows from the induction hypothesis. 
\end{proof}
\begin{lemma}
    Suppose that $n\geq 4$. If a line $L\subset X$ contains an Eckardt point, then it is of type II. 
\end{lemma}
\begin{proof}
    The tangent space of a line $L\in F_1(X,x)$ is isomorphic to $H^0(L, \mathcal{N}_{L/X}(-1))$. In particular, if $L$ is a line of type I, then an irreducible component of $F_1(X,x)$ containing $L$ has dimension at most $h^0(L, \OK_L(-1)^2\oplus \OK^{n-4}_L)=n-4$. However, by Lemma~\ref{Le: DimF1(X,x)Cubic} any irreducible component of $F_1(X,x)$ has dimension $n-3$ if $x$ is an Eckardt point.
\end{proof}
Let $S_2\subset F_1(X)$ be the scheme parameterising lines of type II. 
\begin{lemma}\label{Le: Dim.LinesTypeII}
    Suppose that $\cha(K)\neq 2,3$. Then it holds that $\dim S_2\leq n-3$. 
\end{lemma}
\begin{proof}
    This is proved in Lemma 2.12 of \cite{HuybrechtsCubic}. 
\end{proof}
\begin{corollary}\label{Cor: FinitenessEckardt}
    Let $X\subset \PP^n$ be a smooth cubic hypersurface with $n\geq 4$. Then $X$ contains at most finitely many Eckardt points. 
\end{corollary}
\begin{proof}
    Consider the incidence correspondence \[
    \begin{tikzcd}
    &J=\{(x,L)\in S\times S_2\colon x\in L\}\arrow[dl,"p_1"]\arrow[dr, "p_2"]&\\
    S && S_2
    \end{tikzcd}
    \]
    where we recall that $S\subset X$ is the set of Eckardt points of $X$. By Lemma~\ref{Le: Line.only2Eckardt} we know that $p_2$ has finite fibers and hence $\dim J= \dim S_2\leq n-3$, where the last inequality follows from Lemma~\ref{Le: Dim.LinesTypeII}. In addition, for any $x\in S$ every irreducible component of $p_2(x)^{-1}$ has dimension $n-3$ by Lemma~\ref{Le: DimF1(X,x)Cubic}.  In particular, we obtain $\dim S= \dim J-(n-3)\leq 0$, which is sufficient.
\end{proof}
\section{Kontsevich spaces of stable maps}\label{Sec: Kontsevich}
Let $X\subset \PP^n$ be a closed subvariety over an algebraically closed field $k$. For any integers $e\geq 1$ and $k\geq 0$, the Kontsevich space $\overline{\M}_{0,k}(X,e)$ parameterises the data $(f,C, p_1,\dots, p_k)$, where $C$ is a proper, connected, reduced, at worst nodal curve of arithmetic genus $0$; $p_1,\dots, p_n\in C$ are distinct smooth points of $C$ and $f\colon C\to X$ is morphism of degree $e$ satisfying the stability condition that the normalisation of a component of $C$ contracted by $f$ has at least three special points (which are points above marked points or nodes). Since we do not assume that $\cha(K)=0$, it is not necessarily true that $\bar{\M}_{0,0}(X,e)$ is a Deligne--Mumford stack, because inseparable maps can have non-reduced automorphism groups. Nevertheless, $\bar{\M}_{0,0}(X,e)$ is a proper algebraic stack of finite type with finite diagonal and admits a coarse moduli space $\bar{M}_{0,k}(X,e)$, which is a projective variety~\cite[Theorem 4.7]{deJongStarrRationallysimplyconnected}. It thus admits a finite surjective morphism from a scheme by \cite[Theorem 2.8]{KreschVistoliStackCover}. These are all the properties we need for our applications. 

Our goal is to use the bend-and-break approach invented by Harris, Roth and Starr~\cite{HarrisRothStarrRatCurvesI} to show that under the hypotheses of Theorem~\ref{Th: ModuliSpacescubicQuadric} the spaces $\bar{\M}_{0,0}(X,e)$ are of the expected dimension for any $e\geq 1$. 

Let $\M_{0,k}(X,e)\subset \bar{\M}_{0,k}(X,e)$ denote the open substack parameterising stable maps with irreducible domain and write $\evalu_e^{(k)}\colon \bar{\M}_{0,0}(X,e)\to X\times \cdots \times X$ for the evaluation morphism $(f,C, p_1,\dots, p_k)\mapsto (f(p_1),\dots, f(p_k))$. When $k=1$, we will simply write $\evalu_e$ for $\evalu_e^{(1)}$. 

The key ingredient for the argument is the following simple version of bend-and-break. 

\begin{lemma}\label{Le: BendBreak}
    There is no complete curve in $\M_{0,2}(X,e)$ contained in a fiber of the evaluation morphism $\evalu_e^{(2)}$. 
\end{lemma}
\begin{proof}
    This is Lemma~5.1 of \cite{HarrisRothStarrRatCurvesI}. Going through the argument and the references therein, one sees that the argument does not use that assumption that the characteristic of the ground field is 0. 
\end{proof}

We will use bend-and-break in the form of the following corollary, which is analogous to Proposition~3.4 of \cite{RiedlYang}.
\begin{corollary}\label{Cor: BendBreak}
    Let $x\in X$ and suppose that $Y\subset \evalu_e^{-1}(x)$ is a closed substack of dimension at least $\dim(X)$. Then $Y$ parameterises stable maps with reducible domain and their locus has codimension at most 1 in $Y$. 
\end{corollary}
\begin{proof}
    We begin by showing that $Y$ parameterises maps with reducible domain. 
    Let $Z\subset \bar{\M}_{0,2}(X,e)$ be the preimage of $Y$ under the morphism $\bar{\M}_{0,2}(X,e)\to \bar{\M}_{0,1}(X,e)$, so that $Z$ has dimension at least $\dim(Y)+1$. Consider the morphism $\phi\colon Z\to X$ sending $(f,C, p_1,p_2)$ to $f(p_2)$. Then because $\dim Z >\dim(Y)\geq \dim (X)$, for $y\in X$ that lies in the image of $\phi$, we have $\dim \phi^{-1}(y)\geq 1$. In addition, any $(f,C, p_1,p_2)\in \phi^{-1}(y)$ satisfies $f(p_1)=x$ and $f(p_2)=y$ by construction. Lemma~\ref{Le: BendBreak} thus implies that $Z$ contains maps with reducible domain and hence so does $Y$. 

    Let $\partial Z = Z \cap (\bar{\M}_{0,2}(X,e)\setminus \M_{0,2}(X,e))$ be the locus of reducible maps in $Z$. To prove that the locus of maps with reducible domain inside $Y$ has codimension at most 1, it suffices to prove the same statement for $Z$. To do so, it is enough to show that for any $y\in Y$ in the image of $\phi$, the codimension of $\phi^{-1}(y)\cap \partial Z$ in $\phi^{-1}(y)$ is at most 1. Suppose the contrary holds. Then there exists $y\in X$ such that $\phi^{-1}(y)\cap \partial Z$ has codimension at least 2 in $\phi^{-1}(y)$. Let $S$ be the coarse moduli space of $\phi^{-1}(y)$ inside $\bar{M}_{0,2}(X,e)$ and $\partial S$ be the coarse moduli space of $\phi^{-1}(y)\cap \partial Z$ inside $\bar{M}_{0,2}(X,e)$. Then because $S$ is projective and $\partial S$ has codimension at least $2$ in $S$, we can find a curve $C$ in $S$ that does not intersect $\partial S$. Since $\bar{\M}_{0,2}(X,e)$ admits a finite surjective morphism from a scheme, we can pull-back $C$ to a complete curve $C'\to \phi^{-1}(y)$ whose image does not intersect $\phi^{-1}(y)\cap \partial Z$. However, then the image of $C'$ must be contained in the fiber of the evaluation morphism $\M_{0,2}(X,e)$ above $(x,y)$ and thus Lemma~\ref{Le: BendBreak} implies that $C'\to \phi^{-1}(y)$ is constant, contradicting our construction of $C'$.
\end{proof}

Let $X\subset \PP^{n}$ be a complete intersection defined by forms of degree $d_1,\dots, d_R$ respectively, so that $\dim(X)=n-R$. We will write $d=d_1+\cdots +d_R$ and refer to 
\[
e_0\coloneqq \left\lceil \frac{d-R}{n+1-d}\right\rceil
\]
as the threshold degree of the tuple $(n,d_1,\dots, d_R)$. Note that our threshold degree can be slightly smaller than the one defined in \cite[(54)]{HarrisRothStarrRatCurvesI}, which comes from the fact that we modified the proof structure of the next result, which we can because we are only interested in the dimension and not the irreducibility of the moduli spaces in question. 
\begin{proposition}\label{Prop: Inductive bend and break}
    Let $X\subset \PP^n$ be a complete intersection as above with $n+1-d\geq 2$. Suppose that there exists a finite set $S\subset X$ such that $\dim \evalu_e^{-1}(x)=e(n+1-d)-2+\delta_{x\in S}$ for every $x\in X$ and $1\leq e \leq e_0$. Then the same is true for every $e\geq 1$. 
\end{proposition}
\begin{proof}
    We will prove the result by induction. The case $e\leq e_0$ holds by assumption and we can thus assume that $e>e_0$ and the result is true for every $1\leq e'<e$. 

    Before showing the result, we prove an auxiliary statement. For $x\in X$, let $\partial M_{e'}(x)$ be the locus of maps with reducible domain in $\evalu_{e'}^{-1}(x)$. We will now show that for every $1\leq e'\leq e$, we have $\dim \partial M_{e'}(x)\leq e'(n+1-d)-3+\delta_{x\in S}$. We will again proceed by induction. For $e'=1$ it is clear that $\partial M_{e'}(x)=\emptyset$, hence $\dim \partial M_{e'}(x)=-1 \leq n+1-d-3$, since we assume $n+d-1\geq 2$. 
    
    Let us now assume $e'>1$ and that the result holds for every $1\leq e'' <e'$. Let $N_{e''}(x)\subset \bar{\M}_{0,2}(X,e)$ be the preimage of $\evalu_{e''}^{-1}(x)$ under the morphism forgetting the second marked point. Then we can cover $\partial M_{e'}(x)$ with stacks of the form $\bar{\M}_{0,1}(X,e'-e'')\times_X N_{e''}(x)$ for $1\leq e'' <e'$, where we use the evaluation morphism $\bar{\M}_{0,1}(X,e)\to X$ and the morphism $N_{e''}(x)$ sending $(f,C,p_1,p_2)$ to $f(p_2)$ to define the fiber product. Explicitly this means that we take stable maps $(C,f,p_1,p_2), (C', f', p')$ such that $f(p_1)=x$, $f(p_2)=f'(p')$ and glue them along $p_2$ and $p'$ (see \cite[Proposition~2.4]{BehrendManinStacks} for the gluing construction). 

    Let $W\subset \bar{\M}_{0,1}(X,e'-e'')\times_X N_{e''}(x)$ be an irreducible component, $V\subset N_{e''}(x)$ its image under the projection $\pi \colon \bar{\M}_{0,1}(X,e'-e'')\times_X N_{e''}(x)\to N_{e''}(x)$ and $Y\subset  \evalu_{e''}^{-1}(x)$ the image of $V$ under the morphism $N_{e''}(x)\subset \bar{\M}_{0,2}(X,e')\to \bar{\M}_{0,1}(X,e')$ forgetting the second marked point. Let $(f,C,p_1,p_2)\in V$ be a general map, so that $\dim W = \dim V +\dim \pi^{-1}(f,C,p_1,p_2)$. Also note that $\pi^{-1}(f,C,p_1,p_2)$ can be identified with $\evalu^{-1}_{e'-e''}(f(p_2))$, which together with the first induction hypothesis gives 
    \begin{equation}\label{LALALALA}
    \dim W \leq \dim V + \dim \evalu^{-1}_{e'-e''}(f(p_2)) \leq \dim V + (e'-e'')(n+1-d)-2+\delta_{f(p_2)\in S}. 
    \end{equation}
\noindent
    \textit{Case I:} Assume that $f(p_2)\in S$. Then since $S$ is finite, we have $\dim V=\dim Y$ unless $p_1$ lies on a component of $C$ contracted by $f$, in which case $\dim V=\dim Y+1$. In the latter case $C$ must be reducible, which implies that $Y\subset \partial M_{e''}(x)$. The induction hypothesis now implies that $\dim Y \leq e''(n+1-d) -3+\delta_{x\in S}$ and hence by \eqref{LALALALA}
    \[
    \dim W \leq (e''(n+1-d)-3+\delta_{x\in S})+1 +(e'-e'')(n+1-d)-2+1 = e'(n+1-d)-3+\delta_{x\in S},
    \]
    which is sufficient. 

    If $p_1$ does not belong to a component of $C$ contracted by $f$, then $\dim V= \dim Y$. Since $Y\subset \evalu^{-1}_{e''}(x)$, the first induction hypothesis and \eqref{LALALALA} give 
    \[
    \dim W \leq e''(n+1-d)-2+\delta_{x\in S} +(e'-e'')(n+1-d)-1= e'(n+1-d)-3+\delta_{x\in S},
    \]
    which is again satisfactory.
    
    \noindent
    \textit{Case II:} Assume that $f(p_2)\not\in S$. Here the first induction hypothesis gives 
    \[
    \dim V \leq \dim Y+1 \leq \dim \evalu^{-1}_{e''}(x)+1\leq e''(n+1-d)-1+\delta_{x\in S} 
    \]
    and hence by \eqref{LALALALA}
    \[
    \dim W \leq e''(n+1-d)-1+\delta_{x\in S} +(e'-e'')(n+1-d)-2 = e'(n+1-d)-3+\delta_{x\in S},
    \]
    which completes our verification of the claim that $\dim \partial M_{e'}(x)\leq e'(n+1-d)-3+\delta_{x\in S}$ for every $1\leq e' \leq e$. 

    We can now return to our original goal of showing that $\dim \evalu_e^{-1}(x)\leq e(n+1-d)-2+\delta_{x\in S}$. Suppose this is not true. This means that there exists an $x\in X$ and an irreducible component $Z$ of $\evalu_{e}^{-1}(x)$ such that $\dim Z \geq e(n+1-d)-1+\delta_{x\in S}$. However, this implies 
    \[
    \dim Z \geq (e_0+1)(n+1-d)-1\geq n-R=\dim X, 
    \]
    where $e_0$ was defined in precisely such a way that the second inequality always holds. In particular, Corollary~\ref{Cor: BendBreak} implies that $Z$ contains maps with reducible domain and their locus $\partial Z$ has codimension at most 1 in $Z$. By what we have shown for reducible maps, it follows that 
    \[
    \dim Z \leq \dim \partial Z +1 \leq e(n+1-d)-2+\delta_{x\in S}. 
    \]
    This contradicts our assumption that $\dim \evalu_{e}^{-1}(x) > e(n+1-d)-2-\delta_{x\in S}$ and so $\dim \evalu_{e}^{-1}(x)\leq e(n+1-d)-2+\delta_{x\in S}$ must hold. 
\end{proof}

We can now apply our result to smooth cubic hypersurfaces and intersections of two quadrics, using the results about lines from the previous section. 

\begin{proof}[Proof of Theorem~\ref{Th: ModuliSpacescubicQuadric}]
    Let $X\subset \PP^n$ be either a smooth cubic hypersurface or a smooth intersection of two quadrics of dimension at least 3 and write $d$ for the respective degree.

    We want to combine our findings of Sections~\ref{Se: Linescubic} and \ref{Sec: Kontsevich}. We begin with a few reduction steps. Since any non-empty irreducible component of $\Mor_e(\PP^1,X)$ is of dimension at least $e(n+1-d)+\dim(X)$, we only have to establish the same upper bound.
    Note that since $\M_{0,0}(X,e)$ can be realised as the stack quotient $[\Mor_e(\PP^1,X)/\PGL_2]$, to prove that $\Mor_e(\PP^1,X)$ is of the expected dimension, it suffices to show that $\dim (\M_{0,0}(X,e))\leq e(n+1-d)+\dim(X)-3$. In addition, it reduces to proving that $\dim(\bar{\M}_{0,1}(X,e))\leq e(n+1-d)+\dim(X)-2$, because $\M_{0,0}(X,e)$ is an open substack of $\bar{\M}_{0,0}(X,e)$ and $\dim(\bar{\M}_{0,1}(X,e))=\dim(\bar{\M}_{0,0}(X,e))+1$. This will follow if we can show that there exists a finite set $S\subset X$ such that 
    \begin{equation}\label{Eq: YES}
    \dim(\evalu_{e}^{-1}(x))=e(n+1-d)-2+\delta_{x\in S}
    \end{equation}
    for any $x\in X$ and any $e\geq 1$. 

When $e=1$, this is precisely the content of Propositions~\ref{Prop: LinesIntersectionQuadrics} and \ref{Le: DimF1(X,x)Cubic} with $S$ being the set of Eckardt points when $X$ is a cubic hypersurface, which is finite by Corollary~\ref{Cor: FinitenessEckardt}, and $S$ being the empty set when $X$ is an intersection of two quadrics. In particular, \eqref{Eq: YES} will follow from Proposition~\ref{Prop: Inductive bend and break} if the required hypotheses are satisfied. We clearly have $n+1-d\geq 2$ and 
\[
e_0=\left\lceil \frac{2}{n+1-d}\right\rceil =\begin{cases}
    \left\lceil \frac{2}{n-3}\right\rceil &\text{if }d=4,\\
    \left\lceil \frac{2}{n-2}\right\rceil &\text{if }d=3.
\end{cases}
\]
It is now easy to check that $e_0\leq 1$ when $\dim(X)\geq 3$, which completes our proof. 
\end{proof}
\printbibliography
\end{document}